\documentclass[reqno]{amsart}
\usepackage{amsmath,amssymb,cite}
\usepackage[all]{xy}
\usepackage{tikz}
\usepackage{subfigure}
\usepackage{graphicx}
\usetikzlibrary{graphs, shapes.geometric, arrows.meta}
\usepackage[mathscr]{euscript}
\usepackage{mathtools}
\mathtoolsset{showonlyrefs}
\usepackage{xcolor}
\usepackage{hyperref}
\usepackage[c]{esvect}
\usepackage[normalem]{ulem}

\definecolor{bblue}{rgb}{.2,0.2,.8}

\theoremstyle{plain}
\newtheorem{theorem}{Theorem}[section]
\newtheorem{proposition}[theorem]{Proposition}
\newtheorem{lemma}[theorem]{Lemma}
\newtheorem{corollary}[theorem]{Corollary}

\theoremstyle{definition}
\newtheorem{definition}[theorem]{Definition}
\newtheorem{assumption}[theorem]{Assumption}

\theoremstyle{remark}
\newtheorem{remark}[theorem]{Remark}

\numberwithin{equation}{section}
\numberwithin{theorem}{section}

\newcommand{\upbar}[1]{\,\overline{\! #1}}

\newcommand{\id}{{1 \mskip -5mu {\rm I}}}
\renewcommand{\epsilon}{\varepsilon}
\renewcommand{\hat}{\widehat}

\begin{document}

\title{Freidlin-Wentzell solutions of discrete Hamilton-Jacobi equations} 

\author{Michele Aleandri}
\address{\noindent Michele Aleandri \hfill\break\indent 
	DEF, LUISS University
	\hfill\break\indent 
	00197 viale Romania, Roma,  Italy
}
\email{maleandri@luiss.it}

\author{Davide Gabrielli}
\address{\noindent Davide Gabrielli \hfill\break\indent 
	DISIM, Universit\`a dell'Aquila
	\hfill\break\indent 
	67100 Coppito, L'Aquila, Italy
}
\email{davide.gabrielli@univaq.it}

\author{Giulia Pallotta}
\address{\noindent Giulia Pallotta \hfill\break\indent 
	DISIM, Universit\`a dell'Aquila
	\hfill\break\indent 
	67100 Coppito, L'Aquila, Italy
}
\email{giulia.pallotta@graduate.univaq.it}

\begin{abstract}
	We consider a sequence of finite irreducible Markov chains with exponentially small transition rates: the transition graph is a fixed, finite, strongly connected directed graph; the transition rates decay exponentially on a paramenter $N$ with a given rate that varies from edge to edge. The stationary equation uniquely identifies the invariant measure for each $N$, but at exponential scale in the limit as  $N$ goes to infinity reduces to a discrete equation for the large deviation rate functional of the invariant measure, that in general has not an unique solution. In analogy with the continuous case of diffusions, we call such equation a discrete Hamilton-Jacobi equation. Likewise in the continuous case we introduce a notion of viscosity supersolutions and viscosity subsolutions and give a detailed geometric characterization of the solutions in terms of special faces of the polyhedron of one-Lipschitz functions on the transition graph. This parallels the weak KAM theory in a purely discrete setting. We identify also a special vanishing viscosity solution obtained in the limit from the combinatorial representation of the invariant measure given by the matrix tree theorem. The result gives a selection principle on the set of solutions to the discrete Hamilton-Jacobi equation obtained by the Freidlin and Wentzell minimal arborescences construction; this enlightens and parallels what happens in the continuous setting.
\end{abstract}	

\noindent


\maketitle
\thispagestyle{empty}

\section{Introduction}

The Freidlin and Wentzell theory (FW theory) \cite{FW} is a milestone in the theory of stochastic processes and develops a general framework for studying the small noise limit of diffusion processes. The language and techniques used are typical of Large Deviations theory (LD theory) and are strongly connected to other areas of mathematics, such as discrete mathematics, calculus of variations, and the theory of Hamilton-Jacobi (HJ) equations.

\smallskip
The basic mathematical object of FW theory is a sample-path LD principle for the weak noise limit of a diffusion process. This is a cost functional defined on absolutely continuous trajectories, which plays the role of a classical action obtained by integrating a Lagrangian. Starting from the action, FW theory introduces the so-called {\it Quasipotential}, which is obtained by a suitable minimization of the dynamic action. In cases where the stochastically perturbed vector field has multiple stable minima, the construction of the quasipotential also involves a discrete optimization procedure. The invariant measure of the system satisfies an LD principle with a rate functional that coincides with the quasipotential. By classical variational arguments, the quasipotential solves a stationary HJ equation. The stationary HJ equations arising in the FW theory do not have a unique {\it viscosity solution}, but the uniqueness of the LD rate functional of the invariant measure selects a special solution. This solution can be obtained directly by considering the small noise limit of the unique invariant measure, in the same spirit as the vanishing viscosity limit for HJ equations. These special solutions of the HJ equations, arising in the weak noise limit of diffusions, have been considered, for example, in \cite{FG,GL,GT,J}.

\smallskip

The FW theory has a basic structure that has been extended to generalized frameworks, both infinite-dimensional and finite-dimensional. In the discrete case of continuous-time Markov chains, the analogous framework is that of Markov chains with exponentially small transition rates, which are widely used and considered in statistical mechanics and metastability (see, for example, \cite{OS1, OS2, S1, S2, S3}). For Markov chains of this type, we have a sample-path LD principle and an LD principle for the invariant measure. Within the framework of LD theory, we introduce a discrete HJ equation that must be solved by the LD rate functional of the invariant measure of the chain. We introduce the notion of sub- and super-solutions, which allows us to identify solutions—i.e., functions that are both super- and sub-solutions—as points of the faces of the polytope of one-Lipschitz functions, with the appropriate dimensionality. A special solution is identified by the limit of the unique invariant measure, in the same spirit as the vanishing viscosity solutions. This special solution is constructed through a discrete optimization problem on rooted directed arborescences, which is the same optimization problem used to define the quasipotential in the continuous FW case with several stable equilibria. We note that our HJ equation is defined on any weighted directed graph without any other geometric structure. In particular, there is a natural action defined on paths obtained through an LD principle, but there is no corresponding natural Hamiltonian. It would be interesting to extend our constructions to metric graphs, where continuous arborescences play a role in the identification of the invariant measure, see \cite{ACG}, and a theory for HJ equations has been developed (see for example \cite{CM} and references therein). We also remark that the discrete Lax-Oleinik semigroup, which we introduce using our discrete action, is related to algorithms for shortest paths on directed graphs, such as Dijkstra's or the Bellman-Ford-Moore algorithms, see \cite{JJ}.

In the case of diffusion processes there is a rich and not trivial relation among large deviations and HJ equations, as illustrated in the references \cite{FG,GL,GT,J}, where non uniqueness, non differentiability, phase transitions and selection principles are intertwined. The main aim of this paper is to illustrate and study all these issues in the simplified discrete setting. In the discrete setting we can, for example, obtain a complete geometric characterization of viscosity solutions, which was instead difficult to achieve in \cite{FG}, with the aim of providing a geometric characterization of the selection principle for the vanishing viscosity solution. Moreover, the discrete version of the theory is interesting in its own right, with strong connections to other mathematical problems on graphs and with challenging extensions, such as to the case of metric graphs and related stochastic processes defined on them.

\smallskip

In summary, we introduce a stationary discrete HJ equation on a finite directed weighted graph, which naturally arises from variational problems associated with the LD rate functional of Markov chains with exponentially small transition rates. We provide a detailed analysis of the structure of the viscosity solutions and identify a special solution obtained as the limit of the unique invariant measure. The discrete theory parallels the weak KAM theory and FW theory in the continuous setting. The discrete framework allows for a clearer view of formulas, constructions, and characterizations.

\smallskip
The paper is organized as follows:

Section 2 introduces our framework of continuous-time Markov chains with exponentially small transition rates, describes the basic assumptions, and derives a discrete HJ equation on abstract directed weighted graphs as a natural equation within LD theory.

Section 3 identifies a special solution of the discrete HJ equation using a combinatorial representation of the invariant measure in terms of directed arborescences. We call this solution the {\it Freidlin–Wentzell solution}.

Section 4 introduces the notion of viscosity solutions, describes special features in terms of the graph’s geometry, and demonstrates that the combinatorial solution is indeed a viscosity solution.

Section 5 explores the geometry of the set of viscosity solutions, identifying it with specific faces of the polyhedron of one-Lipschitz functions with particular dimensionality. We also introduce the notion of {\it Quasipotentials} and provide a general representation formula for viscosity solutions.

Section 6 briefly outlines a parallel with the weak KAM theory obtained from a sample path LD principle for the Markov chains.

Section 7 presents examples and identifies the FW solution as the vanishing viscosity solution.

\section{The discrete Hamilton Jacobi equation }\label{sec:frame}

Let $\mathcal{G} = (V, E)$ be a finite, strongly connected directed graph, meaning that any two vertices can be connected by directed paths. Assume there are no loops, i.e., no edges of the form $(x, x)$. Consider a sequence of irreducible Markov chains with transition probabilities depending on a parameter \(N\). Denote by \((r_N(x, y))_{(x, y) \in E}\) the rates of jumps from \(x \in V\) to \(y \in V\) with \((x, y) \in E\), and define \(r_N(x) := \sum_{y : (x, y) \in E} r_N(x, y)\) as the total rate of jumps from \(x \in V\).\\

Under these assumptions, the invariant measure is unique \cite{No}. Furthermore, the stationarity condition characterizing the invariant measure \(\pi_N\) is given by:
\begin{equation}\label{rice}
	\pi_N(x)\sum_{y: (x,y)\in E} r_N(x,y)=\sum_{y:(y,x)\in E} \pi_N(y)r_N(y,x)\,, \qquad \forall x\in V\,.
\end{equation}

We consider the case where the rates are exponentially small. This is a classic framework used, for example, in \cite{FW} to describe the effective fluctuations of diffusion processes and the behavior at very low temperatures of particle systems. It is also a general model of a multiscale Markov chain in a regime on which there are dynamic and static large deviations.  Specifically, we assume that the following condition holds throughout the rest of the paper.

\begin{assumption}[See \cite{OS2}]\label{assumption-1}
	The rates of jump $(r_N(x,y))_{(x,y)\in E}$ are exponentially small with $N$, namely, the following limits exist and are finite \begin{equation*}\Delta(x,y):=-\lim_{N\to +\infty} N^{-1}\log r_N(x,y)\geq 0\,, \qquad \forall(x,y)\in E\,.
	\end{equation*}
\end{assumption}

Define a sequence of functions $W_N$ by the relation $\pi_N(x)=:e^{-NW_N(x)}$, then equation \eqref{rice} uniquely defines $(W_N)_{N\in \mathbb N}$. The study of large deviations of $(\pi_N)_{N\in \mathbb N}$ is related to the study of the limit of the sequence $(W_N)_{N\in \mathbb N}$. The following lemma shows that the limit can be characterized as a solution of a discrete equation.  
\begin{lemma}[HJ equation]\label{lemma:HJe}
	Consider the sequence of probability measures $(\mu_N)_{N\in \mathbb N}$ solutions to \eqref{rice} and assume that there exist the limits
	\begin{align}
		&W(x):=-\lim_{N\to +\infty}N^{-1}\log \mu_N(x)\,, \qquad \forall x\in V\,. \label{ipo}
	\end{align} 
	Then $W$ satisfies the following set of equations 
	\begin{equation}\label{1HJ}
		W(x)+\min_{y:(x,y)\in E}\left\{\Delta(x,y)\right\}=\min_{y:(y,x)\in E}\left\{W(y)+\Delta(y,x)\right\}\,,\qquad \forall x\in V\,.
	\end{equation}
\end{lemma}

\begin{proof}
	Since in \eqref{rice} there are finite sums, considering $\lim_{N\to +\infty}\frac{1}{N}\log(\cdot)$ on both sides, using Assumption \ref{assumption-1}, we obtain the result by the following basic fact. Given a finite collection $\{\left(a^i_N\right)_{N\in\mathbb{N}};i=1,\ldots,k\}$ of sequences such that there exist the limits $r^i:=\lim_{N\to +\infty}\frac{\log a^i_N}{N}$, it holds 
		$$
		\lim_{N\to + \infty}\frac{1}{N}\log\left(\sum_{i=1}^ka^i_N\right)=\max_i r^i\,.
		$$
		 The proof is elementary and this result is sometimes called the basic principle of large deviations ("the largest exponent wins", see for example \cite{dH} formulas I.1, I.2). Then equation \eqref{1HJ} holds.
\end{proof}

This is precisely what occurs in the continuous case of diffusions, where the HJ equation is derived as the vanishing noise limit of the stationary equation, see \cite{FW}. In line with the classical FW theory for diffusions, we refer to the limiting equation \eqref{1HJ} as a \emph{discrete Hamilton–Jacobi equation} (discrete HJ equation).\\

It turns out that, unlike equation \eqref{rice}, the discrete HJ equation does not always admit a unique solution, as occurs in the continuous setting. An interesting feature of the discrete framework is that it allows for a detailed characterization of the set of solutions. To achieve this, we introduce an additional assumption.


\begin{figure}
	\centering
	\subfigure[A graph $(V,E)$ satisfying Assumption \ref{assumption}. The edges with $\Delta$ equal to zero are highlighted in red.]{
		\centering
		\xygraph{
			!{<0cm,0cm>;<1cm,0cm>:<0cm,1cm>::}
			!{(0,0) }*+{\bullet_{a}}="a"
			!{(2,1.3) }*+{\bullet_{b}}="b"
			!{(4,0) }*+{\bullet_{c}}="c"
			!{(3,-1.3) }*+{\bullet_{d}}="d"
			!{(1,2.4) }*+{\bullet_{e}}="e"
			!{(3,2) }*+{\bullet_{f}}="f"
			!{(4.2,1.6) }*+{\bullet_{g}}="g"
			!{(0.3,-1.7) }*+{\bullet_{h}}="h"
			!{(6,0) }*+{\bullet_{\alpha}}="x"
			!{(8.8,1.3) }*+{\bullet_{\beta}}="y"
			!{(10,0) }*+{\bullet_{\gamma}}="z"
			!{(7,-1.4) }*+{\bullet_{\delta}}="i"
			!{(7.1,1.8) }*+{\bullet_{\epsilon}}="j"
			!{(6.3,2.1) }*+{\bullet_{\zeta}}="k"
			!{(11,2) }*+{\bullet_{\eta}}="l"
			"a":@[red]@/^0.5cm/"b"
			"b":@[red]@/^0.4cm/"c"
			"c":@[red]@/^0.3cm/"d"
			"d":@[red]@/^0.6cm/"a"
			"f":@/^0.3cm/"e"
			"f":@[red]@/^0.3cm/"g"
			"g":@[red]@/^0.4cm/"c"
			"h":@[red]@/^0.3cm/"a"
			"a":"h"
			"d":@/_0.6cm/"x"
			"k":@[red]@/_0.5cm/"f"
			"x":@[red]@/^0.5cm/"y"  
			"y":@[red]@/^0.4cm/"z" 
			"z":@[red]@/^0.7cm/"i"
			"i":@[red]@/^0.3cm/"x"
			"x":@/^0.25cm/"j"
			"x":@/^0.3cm/"k"
			"z":@/^0.3cm/"l"
			"l":@[red]@/_0.4cm/"j"
			"j":@[red]@/^0.3cm/"y"
			"e":@[red]@/_0.4cm/"a"}
		} 
	\subfigure[The corresponding subgraph $(V,E_0)$. ]{
		\centering
		\mbox{ \xygraph{
				!{<0cm,0cm>;<1cm,0cm>:<0cm,1cm>::}
				!{(0,0) }*+{\bullet_{a}}="a"
				!{(2,1.3) }*+{\bullet_{b}}="b"
				!{(4,0) }*+{\bullet_{c}}="c"
				!{(3,-1.3) }*+{\bullet_{d}}="d"
				!{(1,2.4) }*+{\bullet_{e}}="e"
				!{(3,2) }*+{\bullet_{f}}="f"
				!{(4.2,1.6) }*+{\bullet_{g}}="g"
				!{(0.3,-1.7) }*+{\bullet_{h}}="h"
				!{(6,0) }*+{\bullet_{\alpha}}="x"
				!{(8.8,1.3) }*+{\bullet_{\beta}}="y"
				!{(10,0) }*+{\bullet_{\gamma}}="z"
				!{(7,-1.4) }*+{\bullet_{\delta}}="i"
				!{(7.1,1.8) }*+{\bullet_{\epsilon}}="j"
				!{(6.3,2.1) }*+{\bullet_{\zeta}}="k"
				!{(11,2) }*+{\bullet_{\eta}}="l"
				"a":@[red]@/^0.5cm/"b"
				"b":@[red]@/^0.4cm/"c"
				"c":@[red]@/^0.3cm/"d"
				"d":@[red]@/^0.6cm/"a"
				"k":@[red]@/_0.5cm/"f"
				"f":@[red]@/^0.3cm/"g"
				"g":@[red]@/^0.4cm/"c"
				"h":@[red]@/^0.3cm/"a"
				"x":@[red]@/^0.5cm/"y"  
				"y":@[red]@/^0.4cm/"z" 
				"z":@[red]@/^0.7cm/"i"
				"i":@[red]@/^0.3cm/"x"
				"l":@[red]@/_0.4cm/"j"
				"j":@[red]@/^0.3cm/"y"
				"e":@[red]@/_0.4cm/"a"
			}
		}
	}\caption{}\label{fig:Assumption21}
\end{figure}

\begin{assumption}\label{assumption}
	For any $x\in V$, there is exactly one $(x,y)\in E$ such that $\Delta(x,y)=0$. Call $E_0\subseteq E$ the set of edges $(x,y)\in E$ having $\Delta(x,y)$ equal to zero (see Figure \ref{fig:Assumption21}).
\end{assumption}

This is a technical assumption that allows us to give a simple and complete characterization of the solutions of the HJ equation. In particular, the edges in $E_0$ are those across which the chain jumps at zero cost and represent the discrete counterpart of critical points and $\omega$-limit sets. Under this assumption, the geometric structure of $(V,E_0)$ simplifies, and elementary cycles play the same role as stable equilibrium points in continuous dynamics.
In the general case, however, there may be  more variations, making a full characterization of solutions to the discrete HJ equations more challenging. Moreover, the meaning of the assumption is that at each vertex of the graph there is exactly one typical escape direction, corresponding to the unique outgoing edge with zero cost. In absence of fluctuations, the chain follows this unique typical edge, so that our model can be viewed as a perturbation of a discrete deterministic dynamical system (apart from the holding times). We  also exclude cases in which there exists a vertex for which the total jump rate is exponentially small, making it  an attractive point modulo fluctuations.

By Assumptions \ref{assumption-1} and \ref{assumption},  it follows that the function $W:V\to \mathbb R$ of Lemma \ref{lemma:HJe} has to satisfy the following discrete HJ equation 
\begin{equation}\label{1HJ-noabs}
	W(x)=\min_{y:(y,x)\in E}\left\{W(y)+\Delta(y,x)\right\}\,,\qquad \forall x\in V\,.
\end{equation}

\begin{remark}\label{discreteMC}
In this paper we consider continuous-time Markov chains but similar results hold in the case of discrete time Markov chains. We analyze the system at an exponential scale, and to any discrete-time Markov chain with exponentially small transition probabilities one can associate a continuous-time chain having the same asymptotic rates. Some of the related mathematical structure, such as the sample-path large deviations, will have a different form, but the discrete HJ equation will remain the same. Since for a discrete time Markov chain the transition probabilities are normalized to one, this corresponds to the fact that for each $x\in V$ there is at least one edge in $E_0$ exiting from $x$. This, in particular, implies that the HJ equation in the discrete-time case is always of the form \eqref{1HJ-noabs}. The Markov chain tree Theorem also holds in the discrete-time setting, and therefore the discrete-time case corresponds to a special case of the continuous-time one. We will briefly discuss the connection with the discrete weak KAM theory in Section \ref{sec:DweakKAM} in terms of discrete-time Markov chains.
\end{remark}

\section{Freidlin-Wentzell solution}

\subsection{Graphs and notation}

Some general remarkable classes of graphs and their properties are needed in order to discuss solutions to the discrete HJ equation. 

\begin{definition}\label{def-dir-ar}
	Let $\mathcal G$ be a directed graph. An \emph{arborescence} $\tau\subseteq \mathcal G$ directed toward $x\in V$ is a spanning subgraph of $\mathcal G$ such that:
	
	\smallskip
	
	\noindent 1) For each vertex $y\neq x$ there is exactly one directed edge exiting from $y$ and belonging to $\tau$.
	
	\noindent 2) For any $y\in V$ there exists a unique directed path from $y$ to $x$ in $\tau$.
	
	\noindent 3)  There are no edges exiting from $x$.
	
	\smallskip
	
	\noindent Let  $\mathcal T_x$ be the set of \emph{arborescences} of $\mathcal G$ directed toward $x\in V$ and $\mathcal T$ be the set of all directed arborescences $\mathcal T=\cup_{x\in V}\mathcal T_x$.
\end{definition}
If the graph is strongly connected, then clearly $\mathcal T_x$ is not empty for any $x\in V$.

\smallskip

Given a weight function $Q:E\to \mathbb R$, to any arborescence $\tau$ the following weights are associated
\begin{equation}
	\label{wfw}
	\left\{
	\begin{array}{l}
		Q(\tau):=\sum_{e\in \tau} Q(e)\,,\\
		\mathcal P_Q(\tau):=\prod_{e\in \tau} Q(e)\,.
	\end{array}
	\right.
\end{equation}
In the formulas above, the products and sums are taken over all directed edges $e$ that belong to the arborescence $\tau$. For example, if the weight function is $r_N$, then
$\mathcal P_{r_N}(\tau)=\prod_{e\in \tau} r_N(e)$; and if the weight function is $\Delta$, then
$\Delta(\tau)=\sum_{e\in \tau} \Delta(e)$. 
A similar notation is used for subgraphs that are not arborescences.

\smallskip
A cycle in $\mathcal{G}$ is a sequence $(z_1, \dots , z_n)$ of distinct vertices $z_i\in V$ such that $(z_i,z_{i+1})\in E$ where $z_{n+1}\equiv z_1$ .

\subsection{Invariant measures and the Matrix Tree Theorem}
One of the primary goals of this paper is to establish a connection between the large deviations of the sequence $(\pi_N(x))_{x\in V}$, representing the invariant measures of the chains, and the discrete HJ equation. This aims to develop a discrete analogue of the classical FW framework for diffusions \cite{FW}.\\

The unique invariant measure of an irreducible Markov chain has an interesting combinatorial representation.
The Markov chain Matrix Tree Theorem claims that the invariant measure is given by
\begin{equation}\label{eq:FWformula}
	\pi_N(x)=\frac{\sum_{\tau\in \mathcal T_x}\mathcal P_{r_N}(\tau)}{\sum_{\tau\in \mathcal T}\mathcal P_{r_N}(\tau)}\,.
\end{equation}
See, for example \cite{FW} and \cite{PT}, where you can also find an interpretation of \eqref{eq:FWformula} in terms of determinants. For the sake of completeness, the particularly elegant proof of the Markov chain Matrix Tree Theorem from \cite{PT} is outlined in  Appendix \ref{app:MC}.\\		

The Matrix Tree Theorem allows for the immediate deduction of a large deviations principle for the sequence of invariant measures.

\begin{lemma}\label{FWsolemma}
	The sequence of invariant measures $(\pi_N)_{N\in \mathbb N}$ defined in \eqref{eq:FWformula} satisfies a large deviations principle with rate functional
	\begin{equation} \label{FW solution}
		\mathscr{FW}(x):=\lim_{N\to +\infty}-\frac 1N\log \pi_N(x)=\inf_{\tau\in \mathcal T_x} \Delta(\tau)-\inf_{\tau\in \mathcal T} \Delta(\tau)\,, \qquad x\in V\,.
	\end{equation}
	The function $\mathscr{FW}$ is called \emph{$FW$-quasipotential} and it is a solution to  the discrete HJ equation \eqref{1HJ}. 
\end{lemma}
\begin{proof}
	Formula \eqref{eq:FWformula} together with Assumption \ref{assumption-1} implies immediately (again by the basic principle of large deviations "the largest exponent wins" as discussed in the proof of Lemma \ref{lemma:HJe})  that the sequence $(\pi_N)_{N\in \mathbb N}$ satisfies \eqref{FW solution}. The function $\mathscr{FW}:V\to \mathbb R^+$ is the LD rate functional of the invariant measure. By Lemma \ref{lemma:HJe}, $\mathscr{FW}$ is a solution of equation \eqref{1HJ}.
	Note that the additive constant in \eqref{FW solution} implies that $\inf_{x\in V} \mathscr{FW}(x)=0$.    
\end{proof} 

\section{Viscosity solutions}

Observe that the discrete HJ equations \eqref{1HJ} and \eqref{1HJ-noabs} form a collection of equations that must be satisfied at each node of the graph. In general, this discrete HJ equation does not have a unique solution.

\smallskip

Solutions to \eqref{1HJ-noabs} can be naturally characterized, similarly to viscosity solutions for continuous HJ equations, using subsolution and supersolution conditions. Differently from the continuous setting where the notion of viscosity solutions introduces a class of weak solutions that do not satisfy the equation in classical sense, in the discrete setting we are not enlarging the set of solutions of \eqref{1HJ-noabs}, that is algebraic in nature and has a well defined set of solutions. What happens is that, in order to have a characterization of the solutions of \eqref{1HJ-noabs} it is very useful to introduce a notion of subsolution and one of supersolution. The sets of subsolutions and supersolutions are easier to be characterized and the set of solutions is the intersection of the two. The definitions of subsolutions and supersolutions parallel those in the continuous setting and, indeed, are very useful for comparison principles and scaling limits.

\begin{definition}\label{def:inout-comp}
	A function $W:V\to\mathbb{R}$ is a \emph{subsolution} of equation \eqref{1HJ-noabs} if the following inequality holds:
	\begin{equation}\label{subsol}
		W(x)\leq W(y)+\Delta(y,x)\,, \qquad \forall(y,x)\in E\,.
	\end{equation}

	A function $W:V\to\mathbb{R}$ is a \emph{supersolution} of \eqref{1HJ-noabs} if, for any $x\in V$, there exists $y^*=y^*(x)\in V$ such that $(y^*,x)\in E$ and 
	\begin{equation}\label{skeleton}
		W(x)=W(y^*)+\Delta(y^*,x)\,.
	\end{equation}
	Let $E_W\subseteq E$ be the set of edges where equality \eqref{skeleton} is verified and we call it the \emph{Skeleton} of $W$.
	A function $W$ is a \emph{solution} of \eqref{1HJ-noabs} if it is both a subsolution and a supersolution.

	Call $\underline {\mathcal W}$ and $\upbar {\mathcal W}$ the set of subsolutions and supersolutions, respectively, and call $\mathcal W:=\underline {\mathcal W}\cap \upbar {\mathcal W}$ the set of solutions.
\end{definition}

Since adding an arbitrary constant to a solution leads to another solution, it is natural to consider the set of normalized solutions, which we refer to as
\begin{equation}
	\mathcal{W}_0:=\left\{W\in \mathcal W : \inf_{x\in V}W(x)=0\right\}\,.
\end{equation}
As already observed in the proof of Lemma \ref{FWsolemma}, the $\mathscr{FW}$ solution belongs to $\mathcal W_0$.\\
\medskip

\subsection{Maps and unicyclic compontents}

The following  classes of subgraphs will be relevant for the characterization of the viscosity solutions of the discrete HJ equation.

\begin{definition}\label{def-dir-cy}
	Let $\mathcal G$ be a directed graph. A \emph{Map} $\mathfrak f$ is a spanning subgraph of $\mathcal G$ such that for each vertex $x\in V$ there exists exactly one edge $(x,y)\in \mathfrak f$ exiting from $x$.
\end{definition}
The name is due to the fact that a map $\mathfrak f$ is encoded by an injective function $f:V\to V$ such that $\mathfrak f=\Big(V,(x,f(x))_{x\in V}\Big)$.

\smallskip

The last definition refers to subgraphs with only one cycle.

\begin{definition}\label{def:inout}
	A subgraph $\mathfrak u$ of a directed graph $\mathcal G$ is called an \emph{inward oriented unicyclic component}, or simply \emph{in-unicyclic component}, if 
	
	\begin{itemize}
		\item there exists a unique directed cycle $C_{\mathfrak u}$ in $\mathfrak u$,
		\item for any vertex $x$ in $\mathfrak u$ that lies outside the cycle $C_{\mathfrak u}$, there is a unique edge in $\mathfrak u$ exiting $x$, and a directed path from $x$ to $C_{\mathfrak u}$.
	\end{itemize}
	A subgraph $\mathfrak o$ of a directed graph $\mathcal G$ is called an \emph{outward oriented unicyclic component}, or simply \emph{out-unicyclic component}, if 
	
	\begin{itemize}
		\item there exists a unique directed cycle $C_{\mathfrak o }$ in $\mathfrak o$,
		\item for any vertex $x$ in $\mathfrak o$ that lies outside the cycle $C_{\mathfrak o}$, there is a unique edge in $\mathfrak o$ entering $x$, and there exists a directed path from $C_{\mathfrak o}$ to $x$. 
	\end{itemize}
	
	A spanning in-(out-)unicyclic component $\mathfrak u(\mathfrak{o})$, i.e. one that contains all the vertices in $V$,  is  called a \emph{in-(out-)unicyclic graph}. Call $\mathcal U^{\mathrm{in}}$ and  $\mathcal U^{\mathrm{out}}$ the collection of all the in-unicyclic and out-unicyclic graphs, respectively. 
	On a strongly connected graph $\mathcal{G}$, every node belongs to the cycle of at least one in-unicyclic graph and at least one out-unicyclic graph  of $\mathcal{G}$. Call $\mathcal U_x^{\mathrm{in}}$ and $\mathcal U_x^{\mathrm{out}}$ the set of in-unicyclic and out-unicyclic graphs with $x$ belonging to the unique cycle.
\end{definition}


\begin{figure}\label{fig:cycleroot}
	\centering
	\mbox{ \xygraph{
			!{<0cm,0cm>;<1cm,0cm>:<0cm,1cm>::}
			!{(0,0) }*+{\bullet_{a}}="a"
			!{(2,1.3) }*+{\bullet_{b}}="b"
			!{(4,0) }*+{\bullet_{c}}="c"
			!{(3,-1.3) }*+{\bullet_{d}}="d"
			!{(1,2.4) }*+{\bullet_{e}}="e"
			!{(3,2) }*+{\bullet_{f}}="f"
			!{(4.2,1.6) }*+{\bullet_{g}}="g"
			!{(0.3,-1.7) }*+{\bullet_{h}}="h"
			!{(6,0) }*+{\bullet_{\alpha}}="x"
			!{(8.8,1.3) }*+{\bullet_{\beta}}="y"
			!{(10,0) }*+{\bullet_{\gamma}}="z"
			!{(7,-1.4) }*+{\bullet_{\delta}}="i"
			!{(7.1,1.8) }*+{\bullet_{\epsilon}}="j"
			!{(6.3,2.1) }*+{\bullet_{\zeta}}="k"
			!{(11,-1) }*+{\bullet_{\eta}}="l"
			"a":@/^0.5cm/"b"
			"b":@/^0.4cm/"c"
			"c":@/^0.3cm/"d"
			"d":@/^0.6cm/"a"
			"e":@/_0.3cm/"a"
			"f":@/^0.3cm/"g"
			"g":@/^0.4cm/"c"
			"h":@/^0.3cm/"a"
			"x":@/^0.5cm/"y"  
			"y":@/^0.4cm/"z"  
			"z":@/^0.7cm/"i"
			"i":@/^0.3cm/"x"
			"x":@/^0.25cm/"j"
			"x":@/^0.3cm/"k"
			"z":@/^0.3cm/"l"
		}
	}
	\caption{An in-unicyclic component on the left and an out-unicyclic component on the right}
\end{figure}

The structure of maps is well known; refer, for example, to \cite{P} for the following basic fact: every map is the spanning disjoint union of in-unicyclic components. In particular, every in-unicyclic graph is a map.\\

Moreover, note that if the orientation of all the edges in a map is reversed, the resulting graph will have exactly one edge entering each node, forming a spanning disjoint union of out-unicyclic components; conversely, for any spanning collection of disjoint out-unicyclic components, reversing the orientation of its edges produces a map.\\

By Assumption \ref{assumption}, the graph $(V,E_0)$ is a map. Call $\mathcal{C}$ the set of cycles of the map $(V,E_0)$, then $1\leq |\mathcal{C}|\leq \left\lceil\frac{|V|}{2}\right\rceil$ ($\lceil\cdot\rceil$ represents the integer part). The value one is obtained in the case of one in-unicyclic graph and the second bound follows since there are no loops (edges of the type $(x,x)$).

For any $C_i\in \mathcal C$, $i=1,\dots, |\mathcal C|$, call $\mathfrak u_i$ the associated in-unicyclic component of $(V,E_0)$ whose vertices are the \emph{basin of attraction of $C_i$}, i.e. the set of $x\in V$ such that there is a path in $(V,E_0)$ from $x$ to $C_i$. For simplicity of notation we write also $i\in\mathcal C$ to denote the cycle $C_i$.

\medskip

\subsection{Charachterization of $\underline {\mathcal W}$ and  $\upbar{\mathcal W}$}

\subsubsection{The set of subsolutions  $\underline {\mathcal W}$}  A discrete pseudo quasimetric is a function $d:V\times V\to \mathbb [0,+\infty)$ that satisfies
\begin{itemize}
	\item triangle inequality: $d(x,z)\leq d(x,y)+d(y,z)$, for any $x,y,z\in V$,
	\item point inequality: $d(x,x)=0$, for any $x\in V$. 
\end{itemize}
It may be not symmetric, i.e. there may exist $x,y\in V$ such that $d(x,y)\neq d(y,x)$, and not separable, i.e. there may exist $x,y\in V$, with $x\neq y$, such that $d(x,y)=0$.\\

The weight function $\Delta: E\to \mathbb [0,+\infty)$ induces on the graph $(V,E)$ a pseudo quasimetric $d$ defined as
\begin{equation}\label{defd}
	d(x,y):=\inf_{\gamma_{x,y}\in \mathcal{R}_{x,y}}\Delta(\gamma_{x,y})\,,
\end{equation}
where $\mathcal{R}_{x,y}$ is the set of directed paths $\gamma_{x,y}$ from $x$ to $y$: a directed path from $x$ to $y$ is given by $\gamma_{x,y}:= (x_0, x_1, \dots x_n)$ such that $x_i\in V$,  $(x_i,x_{i+1})\in E$, for all $i$ and  $x_0=x, x_n=y$. Call  $\Delta(\gamma_ {x,y})$ the length of the path $\gamma_ {x,y}$ and observe that $d(x,y)=d(y,x)=0$ when $x,y$ belong to the same cycle in the map $(V,E_0)$.

Given $T,S\subseteq V$,  define 
\[d(T,y):=\inf_{x\in T}d(x,y), \quad d(x,S):=\inf_{y\in S}d(x,y),\quad d(T,S):=\inf_{x\in T,y\in S}d(x,y) \]
the distance between a set and a vertex, a vertex and a set, and two sets, respectively.


Call respectively $\bar \gamma_{x,y}, \bar \gamma_{T,y}, \bar \gamma_{T,S}$ the geodetic paths  achieving the minimum in the above definitions (one of the paths chosen arbitrarily in the case of several minimizers).

Given a vertex $x$, it belongs to the exiting basin of attraction of the cycle $C_i\in \mathcal C$
if $d(C_i,x)\leq d(C_j,x)$ for any $j\neq i$. In the case where there is equality among two or more cycles, the node  $x$  is assigned to the basin of attraction of an arbitrary cycle among the minimizers.

\begin{definition}
	Call $\mathrm{Lip}_1$  the set of one-Lipschitz functions with Lipschitz constant one that are identified by the set of $|V|\times (|V|-1)$ inequalities
	\begin{equation}\label{lipW}
		W(x)-W(y)\leq d(y,x)\,, \qquad \forall x,y\in V\,.
	\end{equation}
\end{definition}

The following result links subsolutions and one-Lipschitz functions.

\begin{lemma}\label{lometto}
	The set of subsolutions and the set of one-Lipschitz functions define the same polyhedron, i.e. $\underline {\mathcal W}=\mathrm{Lip}_1$.
\end{lemma}
\begin{proof}
	Let $W$ be a subsolution and let $\gamma_{yx}=(y=x_0,x_1,\dots , x_n=x)$ be a path from $y$ to $x$. Observe that, for any $i$, 
	$W(x_{i+1})-W(x_i)\leq \Delta(x_i,x_{i+1})$ and the left hand side of the sum over $i$ of this inequalities is telescopic.  Inequality \eqref{lipW} follows by taking the infimum over all paths. Conversely if \eqref{lipW} holds then, for any $(y,x)\in E$,
	\[
	W(x)-W(y)\leq d(y,x)\leq \Delta(y,x).
	\]
\end{proof}
Since a function in $\mathrm{Lip}_1$ must be constant along the cycles in $\mathcal C$, the space  $\mathrm{Lip}_1$ is a $|V|-\sum_{C_i\in \mathcal C}(|C_i|-1)$ dimensional polyhedron.\\

\begin{remark}
	Every polyhedron $\mathfrak{P}$ can be written as the Minkowski sum $\mathfrak{P}=\mathfrak{C}+\mathfrak{B}$, where $\mathfrak{C}$ is a convex cone and $\mathfrak{B}$ is a polytope, i.e. a bounded polyhedron, see for example Section 5.3 of \cite{Sch}. In the case of $\mathrm{Lip}_1$, the cone $\mathfrak{C}$ is reduced to a single line corresponding to functions that are multiple of the identity. This characteristic is also inherited by the solutions of the discrete HJ equations. For this reason, we focus solely on the normalized solutions by adding a suitable constant function. We do not delve further into these discrete geometric features.
\end{remark}

\subsubsection{The set of super solutions $\upbar{\mathcal W}$ }
Consider an out-unicyclic component $\mathfrak o$ such that its unique cycle is one of the cycles in $\mathcal C$ and define the function $\omega_{\mathfrak o}:V\to \mathbb{R}$ as follows:
\begin{itemize}
	\item $\omega_{\mathfrak o}(x)=0$ for any $x$ in the unique cycle,
	
	\item $\omega_{\mathfrak o}(x)=0$ for any $x\not \in \mathfrak o$,
	
	\item the values of $\omega_{\mathfrak o}(x)$ for the remaining $x\in \mathfrak o$ are uniquely determined by  the condition $\omega_{\mathfrak o}(a)-\omega_{\mathfrak o}(b)=\Delta(b,a)$, for any $(b,a)\in \mathfrak o$. 
\end{itemize}
The third condition uniquely identifies the values $\left(\omega_{\mathfrak o}(x)\right)_{x\in \mathfrak o}$ since for each $x\in \mathfrak o$ there is a unique path $\gamma$ going from the cycle to $x$ and then $\omega_{\mathfrak o}(x)=\sum_{e\in \gamma}\Delta(e)$.
Define also the indicator function of the nodes belonging to an out-cyclic component $\id_{\mathfrak o}:V\to\{0,1\}$ by fixing $\id_{\mathfrak o}(x)=1$ when $x\in \mathfrak o$ and zero otherwise.

A collection $(\mathfrak o_i)_{i\in I}$ of out-unicyclic components is called {\it spanning} if the cycle of each component $\mathfrak o_i$ is an element of $\mathcal C$, all the components are disjoint and each $x\in V$ belongs to one component.
\begin{proposition}\label{prop:super}
	A function $W:V\to \mathbb R$ is an element of $\upbar {\mathcal W}$, the set of supersolutions of the Hamilton Jacobi equation \eqref{1HJ-noabs}, if and only if there exists a spanning collection of out-unicyclic components $(\mathfrak{o}_{i})_{i\in I}$ and some values of the parameters $(\lambda_{i})_{i\in I}$, with $\lambda_i\in \mathbb R$, such that
	\begin{equation}\label{supersol}
		W=\sum_{i\in I} \left(\omega_{\mathfrak o_i}+\lambda_i \id_{\mathfrak o_i}\right)\,. 
	\end{equation} 
	
\end{proposition}
\begin{proof}
	Consider a function $W$ as in \eqref{supersol}. Note that the terms of the sum are functions with disjoint support. Since the collection $(\mathfrak o_i)_{i\in I}$ is spanning, any $x\in V$ belongs to exactly one $\mathfrak o_i$ and for each $x\in \mathfrak o_i$ there is exactly one $(y,x)\in \mathfrak o_i$. By the disjoint support property we have $W(x)-W(y)=\omega_{\mathfrak o_i}(x)-\omega_{\mathfrak o_i}(y)=\Delta(y,x)$, so that the supersolution condition is satisfied at each $x\in V$.
	
	\smallskip
	
	Conversely suppose that $W$ is a supersolution.  In the skeleton $E_W$ there is at least one edge entering on each vertex of $V$. For each vertex in $V$ remove arbitrarly from $E_W$ all but one entering edge. As discussed below definition \ref{def:inout}, the resulting graph is a  spanning collection of out-unicyclic components $(\mathfrak o_i)_{i\in I}$. The corresponding cycles can be only the elements $C_i\in\mathcal{C}$ and there are no other cycles, indeed $W(x)=W(y)$ for each $x,y\in C_i$. Define $\lambda_i:=W(x)$, $x\in C_i$, $\forall i\in I$. 
	Then, for vertex $z\in \mathfrak o_i$,  $W(z)=\omega_{\mathfrak o_i}(z)+\lambda_i$ and since the supports are disjoint formula \eqref{supersol} holds.
\end{proof}

\subsection{Freidlin-Wentzell solutions}

Lemma \ref{FWsolemma} states that $\mathscr{FW}$ is a solution of the discrete HJ equation \eqref{1HJ-noabs}. It is also interesting and non-trivial to deduce this fact directly by inserting $\mathscr{FW}$ in \eqref{1HJ}.
\begin{proposition}\label{propFW}
	The $FW$-quasipotential in \eqref{FW solution} is always an element of $\mathcal W_0$, the set of normalized solutions of \eqref{1HJ}.
\end{proposition}
\begin{proof}
	The validity of \eqref{1HJ} is checked for each node. Fix $x\in V$ and let $y^*\in V$ be such that $(x,y^*)\in E_0$ and $\tau^*\in \mathcal T_x$ be the minimal arborescence in \eqref{FW solution}. Then $\Delta(x,y^*)=0$ and $\Delta(\tau^*)$ is minimal in $\mathcal T_x$. 
	
	From the setup,  $\gamma^*:=(x,y^*)\cup \tau^* \in  \mathcal{U}^{\mathrm{in}}_x$ and moreover 
	\begin{equation}\label{fox}
		\Delta(\gamma^*)=\min_{\gamma\in \mathcal{U}^{\mathrm{in}}_x} \Delta(\gamma)=\Delta(\tau^*)\,.
	\end{equation} 
	In fact,  any $\gamma\in \mathcal{U}^{\mathrm{in}}_x$ can be written as $\gamma=(x,y)\cup \tau$ for a suitable $y\in V$ and $\tau\in \mathcal T_x$ and $\Delta(\gamma)=\Delta(x,y)+\Delta (\tau)$. Since both $\Delta(x,y^*)$ and $\Delta (\tau^*)$ are minimal (respectively on $(x,y)\in E$ and $\mathcal T_x$) we  proved \eqref{fox}.
	
	For any $y\in V$ , up to a constant, $\mathscr{FW}(y)=\Delta(\tau)$ for a suitable $\tau\in \mathcal T_y$ and moreover when $(y,x)\in E$ the graph $(y,x)\cup \tau\in \mathcal{U}^{\mathrm{in}}_x$. By \eqref{fox} 
	\begin{equation}\label{fox2}
		\mathscr{FW}(y)+\Delta(y,x)=\Delta((y,x)\cup \tau)\geq \Delta(\gamma^*)=\mathscr{FW}(x)\,,\qquad \forall (y,x)\in E\,,
	\end{equation}
	that means $\mathscr{FW}\in \underline{\mathcal W}$.
	It remains to prove that there exists an $(y',x)\in E$ such that the equality holds in \eqref{fox2}, i.e. $\mathscr{FW}\in \upbar{\mathcal W}$. Consider $(y',x)$ the unique edge entering in $x$ and belonging to the unique cycle in
	$\gamma^*\in \mathcal{U}^{\mathrm{in}}_x$ (recall that $\gamma^*$ is carachterized by \eqref{fox}). 
	
	Since $\gamma^*\setminus (y',x)\in \mathcal T_{y'}$
	\begin{equation}
		\mathscr{FW}(y')+\Delta(y',x)\leq \Delta(\gamma^*\setminus (y',x))+\Delta(y',x)=\Delta(\gamma^*)=\mathscr{FW}(x)\,.
	\end{equation}
	Last inequality  together with the inequality \eqref{fox2}, implies the equality.
\end{proof}

\section{The geometric structure of $\mathcal W$}

In this section, we give some general representations of the set of solutions of the discrete HJ equation and explore their geometric characterizations. 

\subsection{Minimal solutions}

Proposition \ref{propFW} establishes that there exists at least one solution of \eqref{1HJ-noabs}. The following proposition indicates that if there are two or more solutions, they can be combined to form others.
\begin{proposition}\label{prop:infsol}
	Suppose that $\left\{W_j\right\}_{j\in J} \subseteq \mathcal W$ is a family of solutions of the discrete HJ equation \eqref{1HJ-noabs}. Then the function $W_J$, defined as 
	\begin{equation*}
		W_{J}(x):=\inf_{j\in J} \{W_j(x)\},\quad \forall x\in V,
	\end{equation*} 
	is a solution of the discrete HJ equation. 
	If moreover $\left\{W_j\right\}_{j\in J} \subseteq \mathcal{W}_0$ then $W_J\in \mathcal{W}_0$,
\end{proposition}
\begin{proof}
	For all $x\in V$:
	\begin{align*}
		\min_{y:(y,x)\in E}\left\{W_J(y)+\Delta(y,x)\right\}&= \min_{y:(y,x)\in E}\left\{\inf_{j\in J}\Big\{W_j(y)\Big\}+\Delta(y,x)\right\} \\
		&= \min_{y:(y,x)\in E}\left\{\inf_{j\in J}\Big\{W_j(y)+\Delta(y,x)\Big\}\right\}\\
		&= \inf_{j\in J}\left\{\min_{y:(y,x)\in E}\Big\{W_j(y)+\Delta(y,x)\Big\}\right\}\\
		&= \inf_{j\in J}\Big\{W_j(x)\Big\}\\
		&= W_J(x).
	\end{align*}
	The forth equality is obtained using that $W_j$ is a solution for each $j\in J$.
	The second statement follows by the fact that $0\leq \inf_x W_J(x)\leq \inf_x W_j(x)=0$, so that $W_J\in \mathcal{W}_0$.
\end{proof}

\begin{corollary}
	The discrete HJ equation \eqref{1HJ-noabs} admits, among the normalized solutions, a minimal one, defined as 
	\begin{equation}
		W_{m}:=\inf \left\{W: W\in \mathcal{W}_0\right\}.
	\end{equation}
\end{corollary}
\begin{proof}
	By Proposition \ref{propFW}, $W_m \in \mathcal{W}_0$ and, by definition, $W_m\leq W$ for any $W\in \mathcal{W}_0$.
\end{proof}

Proposition \ref{prop:infsol} can be easily extended to the solutions of the general discrete HJ equation \eqref{1HJ}.

\subsection{Quasipotentials}

For each subset $C$ of $V$, the distance between set $C$  and a node $x$ defines a function that measures the minimum cost to move from a node in $C$ to the node $x$. If  $C$ is a cycle in $\mathcal{C}$  this function is a normalized solution to the discrete HJ equation.

\begin{definition}
	For each cycle $C\in\mathcal{C}$ of $(V,E_0)$, define the corresponding \emph{quasipotential} based on $C$ as
	\begin{equation}\label{quasiM}
		W_{C}(x):= d(C,x)\,, \qquad x\in V\,.
	\end{equation}
\end{definition}

Observe that, by definition, $W_{C}(x)=0$ for any $x\in C$ and $W_{C}(x)>0$ for any $x\not \in C$. \\

Fix $C\in\mathcal{C}$ and consider the edges $\bar \gamma_{C}:=\cup_{\left\{x\in V\cap C^c\right\}}\bar \gamma_{C,x}$.  Call $(V, \bar\gamma^*_{C})$, the directed graph obtained removing arbitrarily from  $\bar \gamma_{C}$, for each node $x\not\in C$ all the entering edges except one and finally adding the edges of the cycle $C$.

\begin{lemma}\label{lemmaC}
	The directed graph $(V, \bar\gamma^*_{C})=\mathfrak o$ is a spanning out-unicyclic graph with unique cycle $C$. Moreover, $W_C(x)=\omega_{\mathfrak o}(x)$.
\end{lemma}
\begin{proof}
	Since it is constructed using geodetic, the directed graph $(V, \bar\gamma^*_{C})$ has a unique cycle that is $C$ 
	and, for each vertex $x\in V$ outside the cycle, there is exactly one edge entering to $x$, and this means that it is a spanning out-unicyclic component. The equality $W_C(x)=\omega_{\mathfrak o}(x)$ follows by definition.
\end{proof}

\begin{proposition}
	For each $C\in\mathcal{C}$ the \emph{quasipotential} $W_{C}$ based on $C$ belongs to $\mathcal{W}_0$.
\end{proposition}

\begin{proof}
	By Proposition \ref{prop:super} and Lemma \ref{lemmaC}, $W_C$ is a supersolution.
	For the subsolution condition, suppose that there exists $(y,x)\in E$ such that $W_{C}(x)>W_{C}(y)+ \Delta(y,x)$. By definition of $W_C$, if at least one among $x,y$ belongs to $C$, the previous inequality is impossible; we assume therefore that $x,y\not\in C$. Define the path $\tilde{\gamma}_{C,x}=\bar{\gamma}_{C,y}\cup(y,x)$. Then it follows that \[d(C,x)=\Delta(\bar{\gamma}_{C,x})=W_C(x)>W_{C}(y)+ \Delta(y,x)=\Delta\big(\tilde{\gamma}_{C,x}\big)\,,\] 
	which is impossible since $\bar{\gamma}_{C,x}$ is a geodetic from $C$ to $x$. The normalization follows since $W_C\geq 0$ and $W_C(x)=0$ for any $x\in C$.
\end{proof}

This is a special situation of a more general case when, given a collection of spanning out-unicyclic components $(\mathfrak o_i)_{i\in I}$, for each $x\in \mathfrak o_i$ the unique path from  the cycle $C_i$ to $x$ is a geodetic path $\upbar \gamma_{C_i,x}$. An out-unicyclic component of this type is called \emph{geodetic out-unicyclic component}.\\

The following simple lemma will be useful for the characterization.
\begin{lemma}
	Take $V\in\mathcal{W}$. The associated family of spanning out-unicyclic components $(\mathfrak o_i)_{i\in I}$ in Proposition \ref{prop:super} is composed by geodetic out-unicyclic components.
\end{lemma}
\begin{proof}
	Assume by contraddiction that $\mathfrak o_i$ is not geodetic. There exists an $x\in \mathfrak o_i\cap C_i^c$ whose path in $\mathfrak o_i$ from $C_i$ is not the geodetic one. This implies that $V(x)-V(y)>d(C_i,x)$, with $y\in C_i$, and by Lemma \ref{lometto}, $V$ is not a subsolution.
\end{proof}

The following theorem gives a complete characterization of solutions of \eqref{1HJ-noabs} in terms of the parameters $\lambda=(\lambda_i)_{i\in \mathcal C}\in \mathbb R^{|\mathcal{C}|}$.
\begin{theorem}\label{teo:char_alpha}
	The set $\mathcal W$ of the solutions to the discrete HJ equation \eqref{1HJ-noabs} coincides with the family of functions 
	\begin{equation}\label{eq:alphaW}
		W_\lambda=\min_{i\in\mathcal{C}}\Big\{W_{C_i}+\lambda_i\Big\}\,, 
	\end{equation}
	where the vector $\lambda=(\lambda_i)_{i\in\mathcal{C}}$ belongs to the polyhedron $\mathrm{Lip}_1^{\mathcal{C}}$ in $\mathbb{R}^{|\mathcal{C}|}$ defined by
	\begin{equation}\label{ineq:lambda}
		\lambda_i-\lambda_j\leq d(C_j,C_i)\,.
	\end{equation}
\end{theorem}

\begin{proof} 
	For any $i\in\mathcal{C}$, $W_{C_i}+\lambda_i$ is a solution then, by Proposition \ref{prop:infsol}, $W_\lambda$ it is also a solution for any $\lambda\in\mathbb{R}^{|\mathcal{C}|}$. It remains to show that every solution can be written as in formula \eqref{eq:alphaW} with $\lambda$ statisfying \eqref{ineq:lambda}.\\
	
	Take a solution $W\in\mathcal{W}$. Since $W$ is both a supersolution and subsolution, it can be written as \eqref{supersol} with $(\mathfrak{o}_i)_{i\in I}$ spanning geodetic out-unicyclic components. Since any solution has to be constant on each cycle in $\mathcal C$, adding all the edges of the cycles in $\mathcal C$ we can assume that the set of labels $I$ of the components concides with $\mathcal C$. We call $\lambda_i = W (x)$, where $x\in C_i$. Inequality \eqref{ineq:lambda} is then satisfied because $W$ is in $\mathrm{Lip}_1$:
	\begin{equation*}
		\lambda_i-\lambda_j=W(x)-W(y)\leq d(y,x)=d(C_j,C_i),
	\end{equation*}
	where $x\in C_i$ and $y\in C_j$. By definition we have $W(x)=W_{C_i}(x)+\lambda_i$, for any $x\in\mathfrak{o}_i$. The final step in the proof is to show that,  for any $x\in\mathfrak{o}_i$ and $j\neq i$, it holds:
	$W_{C_i}(x)+\lambda_i\leq W_{C_j}(x)+\lambda_j$. Take $y\in C_j$. Since $W$ is subsolution, $W(x)-W(y)\leq d(C_j,x)$ that is equivalent to $W(x)\leq \lambda_j + W_{C_j}(x)$.

\end{proof}
Observe now that the function $d$ induces a discrete quasimetric among the cycles $\mathcal{C}$ in $(V,E_0)$ defined by $d(C,C')$, for any $C,C'\in\mathcal{C}$.  
Under the Assumption \ref{assumption}, the distance $d(C,C')>0$, for any $C\neq C'$. Condition \eqref{ineq:lambda} says that the parameters $(\lambda_i)_{i\in \mathcal C}$ coincide with the set $\mathrm{Lip}_1^{\mathcal C}$ of the one-Lipschitz functions on $\mathcal C$ with respect to the quasimetric $d$. 

\smallskip

The $\mathscr{FW}$ solution can be identified as a special solution of the form \eqref{eq:alphaW}. Let $\mathscr{FW}^{\mathcal C}$ be the minimal weight of arborescences on the complete graph with vertices $\mathcal C$ and weights defined by $d$. We have the following representation of the FW-quasipotential $\mathscr{FW}$.
\begin{proposition}
	The solution $\mathscr{FW}$ in \eqref{FW solution} corresponds to a solution as in \eqref{eq:alphaW} with $\lambda_i=\mathscr{FW}^{\mathcal C}(C_i)$, i.e. we have
	\begin{equation}\label{minfw}  
		\mathscr{FW}(x) = \min_{C\in\mathcal C}\Big\{W_{C}(x)+\mathscr{FW}^{\mathcal C}(C)\Big\}.
	\end{equation}
\end{proposition}
\begin{proof}
	For $x \in C'\in \mathcal C$, it is easy to show by contraddiction that the minimum in \eqref{minfw} is obtained for $C=C'$. Moreover, by definition $W_{C'}(x)=0$. Thus, up to a constant, the right hand side of \eqref{minfw} coincides with $\mathscr{FW}$ on all the vertices belonging to the cycles in $\mathcal C$. Since by Theorem \ref{teo:char_alpha} the values $\lambda$ completely identify the solution, then the right hand side of \eqref{minfw} coincides with $\mathscr{FW}$.  
\end{proof}
The above result implies that the minimal arborescence $\tau\in \mathcal T_x$ is obtained by joining minimal arborescences among the cycles in $\mathcal C$. It incurs zero cost for paths from vertices to the cycle of the basin of attraction and requires a single geodetic path $\gamma_{C',x}$ for some $C'\in \mathcal C$.

\subsection{Faces of the polyhedron $\mathrm{Lip}_1$}
We describe some features of the polyhedron $\mathrm{Lip}_1$ and identify the solutions
of \eqref{1HJ-noabs} with a subset of the relative boundary of such polyhedron.

\smallskip

The polyhedron $\mathrm{Lip}_1$ is identified by the inequalities \eqref{lipW}, but as shown in Lemma \ref{lometto} we can restrict to the inequalities $W(x)-W(y)\leq \Delta (y,x)$, for $(y,x)\in E$. Furthermore, the inequalities associated to edges such that $\Delta(y,x)>d(y,x)$ are redundant. The relative boundary of $\mathrm{Lip}_1$ is the region where at least one of the effective inequalities is indeed an equality. The relative boundary is the union of faces of different dimensions. 

\smallskip 

A face is thus identified by the edges of $E$ where equality holds, meaning that each face corresponds to a subgraph $\mathfrak f\subseteq \mathcal G$ that shares the same vertices as $\mathcal G$ and has special features. Notably, equality must always be satisfied on the edges of the cycles in $\mathcal C$, ensuring that for each $i$, we have $C_i\subseteq \mathfrak f$. Providing a complete geometric characterization of faces can be challenging. For instance, any directed path within $\mathfrak f$ must be geodetic.  Instead, we focus on identifying specific notable faces.  Consider a function $g$ belonging to a face on the relative boundary of $\mathrm{Lip}_1$, associated with the subgraph $\mathfrak f$. If we ignore the orientation of the edges in $\mathfrak f$, then,  since equality holds along there edges, the values of $g$ on vertices within the same connected component of the unoriented subgraph are fully determined by the value of one single vertex. Consequently, the dimension of the face corresponds to the number of components in the induced subgraph. 
This dimension can be decresead by one considering the invariance of the polyhedron by a global addition.

Consider $\mathfrak f=(\mathfrak u_i)_{i\in \mathcal C}$, the collection of in-unicyclic components that are
the basin of attaction of the cycles in $(V,E_0)$. Note that any path in $\mathfrak f$ is geodetic and moreover, for each edge $e\in \mathfrak f$, $\Delta(e)=0$. The subgraph $\mathfrak f$ has $|\mathcal C|$ components and corresponds to a special $|\mathcal C|$-dimensional face of $\mathrm{Lip}_1$ that contains all the functions $g$ that are constant on each basin of attraction $\mathfrak u_i$. The values $\lambda_i$ that the function $g$ assumes on each basin must satisfy \eqref{ineq:lambda}. We call such face the \emph{minimal $|\mathcal C|$ face}.

Consider now $\mathfrak f=(\mathfrak o_i)_{i\in \mathcal C}$ a spanning collection of out-unyciclic geodetic components. Any subgraph $\mathfrak f$ of this type has $|\mathcal C|$ components and corresponds to a $|\mathcal C|$-dimensional face of $\mathrm{Lip}_1$. A function $g$ that belongs to one of such faces assumes the values $\lambda_i$ on the cycles in $\mathcal C$ and the values on all the remaining
vertices in $\mathfrak o_i$ are completely determined by $(\lambda_i)_{i\in\mathcal C}$. The values $(\lambda_i)_{i\in\mathcal C}$ have to satisfy additional constraints  beyond those in  \eqref{ineq:lambda}, depending on the geometry of $\mathfrak f$. In some cases, there may even be  no valid values for $(\lambda_i)_{i\in\mathcal C}$. The new constraints stem from the fact that, given $(\lambda_i)_{i\in\mathcal C}$, the function is naturally Lipschitz within each component, but it is not guaranteed to be globally Lipschitz. We refer to faces of this type as \emph{maximal $|\mathcal C|$ faces}.

\begin{theorem}
	The set $\mathcal W$ of the solutions of the HJ equation \eqref{1HJ-noabs} coincides with the union of the maximal $|\mathcal C|$-dimensional faces of the polyhedron $\mathrm{Lip}_1$. Moreover for any $\lambda=(\lambda_i)_{i\in\mathcal C}\in \mathrm{Lip}_1^{\mathcal C}$ there exist, among the functions in $\mathrm{Lip}_1$  assuming the values $\lambda_i$ for all $C_i\in \mathcal C$, a minimal function $g_m$  belonging to the minimal $|\mathcal C|$-dimensional face and and a maximal one $g_M$ belonging to a 	maximal $|\mathcal C|$-dimensional face and indeed coincides with $W_\lambda$.
\end{theorem}
\begin{proof}
	
	Given a solution $W\in\mathcal W$, the corresponding skeleton identifies the spanning outgoing geodetic family $(\mathfrak o_i)_{i\in \mathcal C}$ of a $\mathfrak f$ that corresponds to a maximal $|\mathcal C|$-dimensional face of $\mathrm{Lip}_1$ to which $W$ belongs. Conversely, if $W$ belongs to a maximal $|\mathcal C|$-dimensional face determined by a map $\mathfrak f$, then $W$ is a subsolution because it is one-Lipschitz. Moreover, it is also a supersolution with the skeleton corresponding to the edges of $\mathfrak f$.

	For the second part of the statement fix $\lambda\in \mathrm{Lip}_1^{\mathcal C}$. For any function $W\in \mathrm{Lip}_1$, for any $x\in V$ and for  any $i\in \mathcal C$, we have $W(x)-\lambda_i\leq d(C_i,x)$. This implies 
	\[
	W(x)\leq \inf_i \left\{\lambda_i+d(C_i,x)\right\}=W_\lambda(x)\,.
	\]
	
	The unique function $g_m$ that belongs to the minimal $|\mathcal C|$-dimensional face and assumes the values $\lambda$ on the cycles is the function that is constantly equal to $\lambda_i$ on each vertex of $\mathfrak u_i$ for every $i\in \mathcal C$. Since the one-Lipschitz condition requires that any function in $\mathrm{Lip}_1$ assuming the values $\lambda\in \mathrm{Lip}_1^{\mathcal C}$ on the cycles must take a value at least $\lambda_i$ on  $\mathfrak u_i$ for each $i$, the proof is complete.
\end{proof}
An interesting question is to find a useful geometric carachterization of $\mathscr{FW}$ as an element of the relative boundary of $\mathrm{Lip}_1$.

\section{Discrete weak KAM theory}\label{sec:DweakKAM}
In this section, we briefly outline how the above results parallel the general structure of the so-called weak KAM theory (see, for example, \cite{E,F,T}, to which we refer for comparison) in a purely discrete setting.

The key step in establishing this correspondence is to prove a sample path large deviations principle. The corresponding rate functional plays the role of the action. Although it is not natural to identify an Hamiltonian, we can define a discrete-time Lax-Oleinik semigroup and develop the parallel analysis. It turns out that it is easier and more natural to discuss this relation in terms of discrete time Markov chains. As explained in Remark \ref{discreteMC} the associated discrete HJ equation is the same.

\smallskip

We consider a family of discrete time Markov chain $\left(X^N_i\right)_{i\in \mathbb N}$, depending on the parameter $N$, assuming values on $V$ and such that
$$
\lim_{N\to+\infty}-\frac{1}{N}\log P_N(x,y)=\Delta(x,y)\geq 0\,,
$$
where $P_N$ are the transition probabilities. This is exactly the framework of \cite{OS1, OS2, S1, S2, S3}. Consider a finite trajectory $\left(X^N_1, \dots X^N_k\right)$; by the assumption \eqref{assumption} we have that the limiting behavior described by the variables $\left(X^\infty_1, \dots X^\infty_k\right)$ is purely deterministic. When $X^\infty_i=x$ then $X^\infty_{i+1}=y$ where $y$ is the unique element of $V$ such that $\Delta(x,y)=0$. The Markov chains $X^N$ are therefore random perturbations of the deterministic dynamical system $X^\infty$. The behavior of $X^\infty$ is such that, after a finite transient period, the system reaches one of the cycles in $\mathcal C$ and moves then periodically around it. A sample path large deviations is easy to compute and it is one of the main instruments used in \cite{OS1, OS2, S1, S2, S3}. Consider a fixed sequence $x:=\left(x_1,\dots ,x_k\right)\in V^k$, we have then
\begin{align*}
	\lim_{N\to +\infty}-\frac{1}{N}\log \mathbb P\Big((X_1^N, \dots ,X^N_k)=(x_1,\dots ,x_k)\Big) 
	&=\lim_{N\to +\infty}-\frac 1N \sum_{i=1}^k \log P_N(x_i,x_{i+1})\\
	&=\sum_{i=1}^k \Delta(x_i, x_{i+1})\,.
\end{align*}
We therefore conclude that the exponential cost of observing a finite trajectory $x$ is given by an action functional $\mathcal A$ on the set of finite trajectories, defined by
\begin{equation}\label{lagrangiana}
	\mathcal A(x):=\sum_{i=1}^k \Delta(x_i,x_{i+1})\,.
\end{equation}

The functional \eqref{lagrangiana} acts as an action, assigning a cost to every finite trajectory of the system and can be used to naturally define  a corresponding discrete-time Lax-Oleinik semigroup.\\

Given  a function $\mathcal V_0:V\to \mathbb R$ we define the iterated function $\mathcal V_n$ at time $n$ as
\begin{equation}\label{L-O}
	\mathcal V_n(x_n):= \inf_{x_0,\dots ,x_{n-1}} \left\{\mathcal V_0(x_0)+\sum_{i=0}^{n-1}\Delta(x_i,x_{i+1})\right\}=\inf_{y} \left\{\mathcal V_{n-1}(y)+\Delta(y,x_n)\right\}\,.
\end{equation} 
This defines our discrete-time Lax-Oleinik semigroup. The fixed points of \eqref{L-O} are the solutions of the discrete HJ equation \eqref{1HJ-noabs}. This framework is natural for our discrete setting, as there is no straightforward definition of the Hamiltonian corresponding to the action \eqref{lagrangiana}.

Based on the definition \eqref{L-O}, we note that our discrete contructions parallel the weak KAM theory, including discrete versions of all the related machinery and geometric constructions. We do not provide a detailed comparison here.

\section{Examples and remarks}

\subsection{The reversible case}

As an example, we consider the reversible case: in this situation, if $(x,y)\in E$ then also $(y,x)\in E$. The detailed balance characterizing the invariant measure $\pi_N$ in the reversible case is given by 
\begin{equation}\label{detailed}
	\pi_N(x) r_N(x,y)= \pi_N(y)r_N(y,x)\,, \qquad  \forall(x,y)\in E\,.
\end{equation}
The counterpart of Lemma \ref{lemma:HJe} for  equation \eqref{detailed} provides the following conditions:
\begin{equation}\label{eq:HJ-rev}
	W(x)+ \Delta(x,y)=W(y)+\Delta(y,x), \quad  \forall(x,y)\in E.
\end{equation}
Equations \eqref{eq:HJ-rev} have a unique solution (up to an arbitrary additive constant) if the condition, called \emph{gradient condition}, is satisfied: $\Delta(x,y)-\Delta(y,x)=:\delta(x,y)$ is a discrete vector field of gradient type, namely antisymmetric $\delta(x,y)=-\delta(y,x)$. If the rates are reversible (i.e. \eqref{detailed} is satisfied) the gradient condition is satisfied and there is a unique solution to \eqref{eq:HJ-rev}; if the gradient condition is not satisfied there are clearly no solutions to \eqref{eq:HJ-rev}. 

The gradient condition requires that, for any cycle $C=(x_0,\dots, x_n)$, the following  holds: $\sum_{(x,y)\in C}\delta(x,y)=0$. In this case the unique, up to an arbitrary constant, solution to \eqref{eq:HJ-rev} is 
\begin{equation}\label{defx*}
	\mathscr{FW}(x):=\sum_{e\in \gamma_{x^*,x}}\delta(e)\,,\quad \forall x\in V,
\end{equation}
where $x^*$ is an arbitrary fixed node and $\gamma_{x^*,x}$ is also an arbitrary path. The independence of the definition \eqref{defx*} from choice of the initial point and of the arbitrary path follows from classical homological arguments, since $\delta$ is antisymmetric and the sum of the weights along any cycle is zero. Indeed, considering two different paths and reversing the direction of one of them yields a cycle, from which the invariance follows. Changing the initial point merely shifts the values by a global additive constant. We used the notation $\mathscr{FW}$ since \eqref{defx*} coincides, up to a constant, with the FW-quasipotential.
Moreover, $\mathscr{FW}(y)-\mathscr{FW}(x)=\sum_{e\in \gamma_{x,y}}\delta(e)$.

In the reversible case the Assumption \ref{assumption} imposes a strict condition on the graph $(V,E_0)$. 

\begin{lemma}
	Given a sequence of reversible Markov chains, Assumption \ref{assumption} imposes that the map  $(V,E_0)$ contains only cycles of length 2, i.e. $|C|=2$ for any $C\in \mathcal C$.
\end{lemma} 
\begin{proof}
	Consider $C\in \mathcal C$. Recalling the gradient condition and the fact that $\Delta(e)=0$ for $e\in C$, we have
	\[
	0=\sum_{e\in C}\delta(e)=-\sum_{e\in C}\Delta(-e)\,,
	\]
	where $-e$ is the edge obtained from $e$ by switching its orientation. The above equalities can be true if only if $\Delta(e)=\Delta(-e)=0$, for any $e\in C$, but this implies that, for each node in $C$, there are at least two edges in $E_0$ unless $|C|=2$.
\end{proof}

\begin{figure}
	\centering
	\mbox{ \xygraph{
			!{<0cm,0cm>;<1cm,0cm>:<0cm,1cm>::}
			!{(0,0) }*+{\bullet_{a}}="a"
			!{(2,1) }*+{\bullet_{b}}="b"
			!{(4,0) }*+{\bullet_{c}}="c"
			!{(2,-1.3) }*+{\bullet_{x}}="d"
			!{(1,2.4) }*+{\bullet_{y}}="e"
			!{(3,2) }*+{\bullet_{f}}="f"
			!{(4.2,1.6) }*+{\bullet_{g}}="g"
			!{(6,0) }*+{\bullet_{a}}="x"
			!{(8,1) }*+{\bullet_{b}}="y"
			!{(10,0) }*+{\bullet_{c}}="z"
			!{(8,-1.3) }*+{\bullet_{x}}="i"
			!{(7,2.4) }*+{\bullet_{y}}="j"
			!{(9,2) }*+{\bullet_{f}}="k"
			!{(10.2,1.6) }*+{\bullet_{g}}="h"
			"e":@[red]@/^0.4cm/"b"
			"b":@[red]@/^0.4cm/"d"
			"a":@/^0.4cm/"d"
			"f":@/^0.4cm/"b"
			"g":@/_0.3cm/"c"
			"c":@/^0.4cm/"d"
			"y":@[red]@/_0.4cm/"j"
			"i":@[red]@/_0.4cm/"y"
			"x":@/^0.4cm/"i"
			"z":@/^0.4cm/"i"
			"k":@/^0.4cm/"y"
			"h":@/_0.3cm/"z"
		}
	}
	\caption{The switch of orientation from an arborescence oriented toward $x$ to an arborescence oriented toward $y$.}\label{fig:revarb}
\end{figure}

\begin{remark}\label{remarkbij}
	In the reversible case, there is a bijection  between  $\mathcal T_x$ and $\mathcal T_y$ for any $x,y\in V$. The map from elements in $\mathcal T_x$ to $\mathcal T_y$ involves inverting the orientation of the edges in the unique directed path from $y$ to $x$ in $\mathcal T_x$. Similarly, the inverse map involves inverting the orientation of the edges in the unique path from $y$ to $x$ in $\mathcal T_y$ (see Figure \ref{fig:revarb}).
\end{remark}

\begin{proposition}
	A minimal weight arborescences ${\mathcal \tau}^*_x\in {\mathcal T}_x$ can be obtained by another minimal weight arborescences ${\mathcal \tau}^*_y\in {\mathcal T}_y$ inverting the orientation of the edges in the unique directed path from $x$ to $y$ in $\mathcal T_y$. In particular all the minimal weight arborescences $({\mathcal \tau}^*_z)_{z\in V}$ have the same unoriented support, differ just by the orientation and $\Delta\left({\mathcal \tau}^*_y\right)-\Delta\left({\mathcal \tau}^*_x\right)=\sum_{e\in \gamma_{x,y}}\delta(e)$.
\end{proposition}

\begin{proof}
	Consider the bijection introduced in Remark \ref{remarkbij}. Given ${\mathcal \tau}_x\in{\mathcal T}_x$ and ${\mathcal \tau}_y\in{\mathcal T}_y$ two elements linked by the bijection, by definition, we have 
	\begin{equation}\label{unarel}
		\Delta\left({\mathcal \tau}_y\right)-\Delta\left({\mathcal \tau}_x\right)=\sum_{e\in \gamma_{x,y}}\delta(e)=\mathscr{FW}(y)-\mathscr{FW}(x).
	\end{equation}
	This implies that the weights of the arborescences in $\mathcal T_y$ are the same as those in $\mathcal T_x$, up to a global shift, and the minimizers are therefore mapped bijectively. Consequently, every pair of minimizers in $(\mathcal T_z)_{z\in V}$ is related simply by the inversion of some arrows. The last statement follows from the fact that \eqref{unarel} holds for any pair of arborescences related by the bijection.
\end{proof}

In the reversible case, we can characterize the maximal $|\mathcal C|$-dimensional face to which the FW-quasipotential belongs.  Recall that a maximal $|\mathcal C|$-dimensional face is indentified by a spanning collection of geodetic out-unicyclic components. 
\begin{proposition}
	In the reversible case the solution $\mathscr{FW}$ belongs to the maximal $|\mathcal C|$-dimensional face of $\mathrm{Lip}_1$ associated to the spanning out-unicyclic geodetic components obtained reversing all the edges of the map $(V,E_0)$.
\end{proposition}
\begin{proof}
	We need to show that the skeleton of $\mathscr{FW}$ is obtained by reversing the edges in $(V,E_0)$. Recall that $\mathscr{FW}$ solves \eqref{eq:HJ-rev}, then for any $(x,y)\in E$ we have
	\[
	\mathscr{FW}(x)+\Delta(x,y)=\mathscr{FW}(y)+\Delta(y,x)\,.
	\]
	Fix $x\in V$ and consider the infimum over all $y$ such that $(x,y)\in E$ of the left hand side of the above equation.  We have that the minimum is achieved for the unique $(x,y^*)\in E_0$. This leads to  the equality
	$\mathscr{FW}(x)=\mathscr{FW}(y^*)+\Delta(y^*,x)$ which is equivalent to say that $(y^*,x)$ belongs to the skeleton of $\mathscr{FW}$.
\end{proof}

\subsection{The ring}
Consider $\mathcal{G}$ as a finite ring with $k$ nodes, where $V=\mathbb Z/(k\mathbb Z)$ and $E=\{(x,x\pm 1), x\in V\}$. In this case, we have a simple and compact representation of the $\mathscr{FW}$ solution, which shares a  similar structure with the continuous one, see \cite{FG}. Let $\delta:E \to \mathbb R$ be the discrete vector field, i.e. such that $\delta(x,y)=-\delta(y,x)$, defined by $\delta(x,x+1):=\Delta(x,x+1)-\Delta(x+1,x)$, with the remaining values defined by antisymmetry. Define the function $S:\mathbb N \to \mathbb R$ by setting $S(x):=\sum_{y=1}^{x-1}\delta(y,y+1)$. In the previous formulas the function $\delta$ is computed on edges of $\mathcal G$ so that the sums in the arguments are modulo $k$ and $y+k=y$. If $S(k+1):=\sum_{y=1}^{k}\delta(y,y+1)=0$ we have that $S(y+k)=S(y)$ so that $S$ is a periodic function of period $k$, and this corresponds to the reversible case (since the gradient condition is satisfied).

\smallskip

We have
\begin{equation}\label{FW-1d}
	\mathscr{FW}(x)=\min_{x\leq y< x+k}\Big\{S(x)-S(y)-\Delta(y,y+1)\Big\}\,.
\end{equation}
The function $\mathscr{FW}$ introduced in \eqref{FW-1d} is originally defined on $\mathbb N$  but it can be readily shown to be periodic with period $k$, allowing for a natural interpretation as a function on $V$. To prove that \eqref{FW-1d} coincides with \eqref{FW solution} up to a constant, it suffices to add and subtract the constant $\sum_{z=1}^k\Delta(z,z+1)$  on the right hand side of \eqref{FW-1d},  getting that, the right hand side of \eqref{FW-1d} coincides with $\min_{\tau\in \mathcal T_x}\left\{\Delta(\tau)-\mathbb \sum_{z=1}^k\Delta(z,z+1)\right\}$.

In the reversible case, when $S$ is periodic, we may write
\[
\mathscr{FW}(x)=S(x)+\min_{x\leq y< x+k}\Big\{-S(y)-\Delta(y,y+1)\Big\}\,.
\] 
The second term on the right hand side does not depend on $x$ so that $\mathscr{FW}$ coincides with $S$ up to an additive constant.

\smallskip

The minimization formula \eqref{FW-1d} closely resembles its continuous counterpart for diffusions on a circle and admits a straightforward yet insightful geometric interpretation, as discussed in  \cite{FG}. While the continuous case is easier to describe, an analogous construction exists in the discrete setting. Specifically, for a continuous function $f$  defined on the unit circle $\mathbb R/\mathbb Z$, one can define the function $S:\mathbb R\to \mathbb R$ as $S(x):=\int_0^xf(y)dy$. Taking
\begin{equation}\label{suncontinuous}
	\mathscr{FW}(x):=\inf_{y\in [x,x+1]}\Big\{S(x)-S(y)\Big\}\,,
\end{equation}
Theorem 2.3 in \cite{FG} summarizes several key properties of the function $\mathscr{FW}$, including its one-periodicity, meaning it can be viewed as a function on the circle. Moreover, if  $\int_0^1f(y)dy=0$, it coincides with $S$ up to an additive constant. Additionally, it can be explicitly constructed using the so-called \textit{sunshine transformation} (see Figure \ref{sunshine}).
\begin{figure}
	\centering
	\includegraphics[width=10cm]{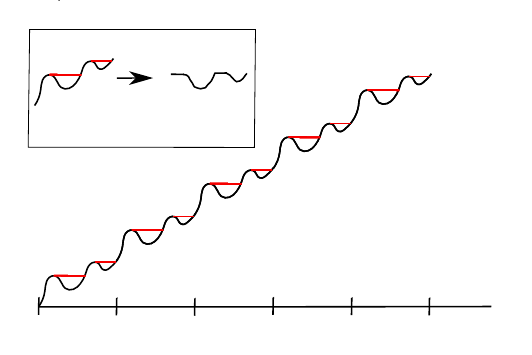}
	\caption{An example of the sunshine transformation for the case of a continuous function $S$ (in black) with $\int_0^1f(y)dy >0$. The light is assumed to arrive horizontally from the left: portions of the graph below the horizontal red line remain in shadow, while the rest are illuminated. The transformation preserves the shadowed regions while replacing the illuminated sections with horizontal segments, shifting them vertically to maintain continuity in the final transformed graph. A visual representation of this process is highlighted within the rectangle in the upper left. The resulting graph is one-periodic. In the case $\int_0^1f(y)dy< 0$, the same construction is applied, but with light arriving from the right. For $\int_0^1f(y)dy=0$, the graph is already periodic, every point lies in shadow, and the transformation leaves the graph unchanged. }\label{sunshine}
\end{figure}
Formula \eqref{suncontinuous} can be interpreted as a minimal continuous arborescences construction, as noted in Remark 1 of \cite{FG}. For definitions of continuous arborescences on metric graphs, as well as a continuous version of the Matrix Tree Theorem for diffusions on metric graphs, we refer to \cite{ACG}. In the case of a continuous ring, a continuos arborescence directed torward a point $x$ is obtained by cutting the ring on another point, say $y$, and orienting the two resulting segments so that they both point torward $x$. Formula \eqref{suncontinuous} coincides up to an additive constant with
\begin{equation}
	\mathscr{FW}(x)=\min_{y\in [x,x+1]}\Big\{\int_x^y[-f(z)]_+dz+\int_y^{x+1}[f(z)]_+dz\Big\}
\end{equation}  
where $[\cdot]_+$ is the positive part. This suggests that solutions of HJ equations associated to diffusions on metric graphs can be obtained by continuous minimal arborescences constructions.

\subsection{Vanishing viscosity solutions}

In the continuous setting, viscosity solutions of HJ equations can be obtained by a vanishing viscosity limit. In the case of a diffusion process on a ring, see for example \cite{FW,FG},  the stationary HJ equation associated to the rate functional of the invariant measure is 
\begin{equation}\label{HJcont}
	H(x,W_x)=W_x(W_x-f(x))=0\,.
\end{equation}
In general, the viscosity solution of \eqref{HJcont} is not unique and a natural selection principle for a solution is to consider the stochastic process behind and look at \eqref{HJcont} as a limit of the stationary Kolmogorov equation. This means to consider the small noise limit $\epsilon \to 0$ of the stationary condition
\begin{equation}
	\epsilon \left(e^{-\epsilon^{-1}W^\epsilon(x)}\right)_{xx}-\left(f(x)e^{-\epsilon^{-1}W^\epsilon(x)}\right)_x=0\,.
\end{equation}
The unique smooth solution $W^\epsilon$ can be found explicitely and converges, in the limit as $\epsilon$ goes to zero, to the solution \eqref{suncontinuous}, i.e. $W^\epsilon\to \mathscr{FW}$. The same solution is selected by the classic vanishing viscosity limit obtained by considering the equation
\begin{equation}
	H(x, W^\epsilon_x)=W^\epsilon_x\left(W^\epsilon_x-f\right)=\epsilon W^\epsilon_{xx}\,.
\end{equation}
Indeed, the unique smooth solution of the previous equation can be computed explicitely by the change of variables $\psi^\epsilon=1/W_x^\epsilon$ that linearizes the equation, and then it can be shown that again $W^\epsilon \to \mathscr{FW}$.

\smallskip

We observe that the solution $\mathscr{FW}$ to the discrete HJ equation \eqref{1HJ-noabs} can also be considered  a \textit{vanishing viscosity solution}. The stationary condition \eqref{rice} written substituting $e^{-NW^N}$ to $\pi_N$ can be written as
\begin{equation}\label{rice2}
	e^{-NW^N(x)}\sum_{y: (x,y)\in E} r_N(x,y)=\sum_{y:(y,x)\in E} e^{-NW^N(y)}r_N(y,x)\,, \qquad \forall x\in V.
\end{equation}
By Lemma \ref{FWsolemma} we have that $\lim_{N\to +\infty} W^N=\mathscr{FW}$.

\subsection{Two dimensional case}\label{exse}
In this subsection we represent graphically the normalized FW-quasipotential as a point in the polyhedron identified by $\mathcal{W}$. We consider the simplest possible case that is special for several reasons but allows a simple picture.\\

Take $V=\{1,2,3,4\}$ and let $E$ be the set of edges depicted in Figure  \ref{fig:2dim}. Suppose that $\Delta(1,2)=\Delta(2,1)=\Delta(3,4)=\Delta(4,3)=0$ and that, without loss of generality,  $0<\Delta(2,3)<\Delta(4,1)$. The graph $(V,E_0)$ contains two cycles, $C_1$ and $C_2$. 
\begin{figure}[h]
	\label{fig:2dim}
	\centering
	\mbox{ \xygraph{
			!{<0cm,0cm>;<1cm,0cm>:<0cm,1cm>::}
			!{(0,0) }*+{\bullet_{1}}="1"
			!{(1,0) }*+{C_1}="5"
			!{(2,0) }*+{\bullet_{2}}="2"
			!{(4,0) }*+{\bullet_{3}}="3"
			!{(5,0) }*+{C_2}="6"
			!{(6,0) }*+{\bullet_4}="4"
			"1":@[red]@/^0.4cm/"2"
			"2":@[red]@/^0.4cm/"1"
			"3":@[red]@/^0.4cm/"4"
			"4":@[red]@/^0.4cm/"3"
			"4":@/^1.5cm/"1"
			"2":@/^0.6cm/"3"
		}
	}\caption{The transition graph of the example in Section \ref{exse}. Transitions corresponding to $\Delta=0$ are drawn in red.}
	\end{figure}  By Theorem \ref{teo:char_alpha} there exists a pair $(\lambda_1,\lambda_2)\in \mathrm{Lip}_1^{\mathcal C}$ such that any solution to the discrete HJ equation \eqref{1HJ-noabs} can be written as
\begin{equation*}
	W_\lambda(x)=\min\{ W_{C_1}(x) + \lambda_1; W_{C_2}(x) + \lambda_2\},
\end{equation*}
where 
\begin{equation*}
	\begin{array}{r|cccc}
		& 1 & 2 & 3 & 4 \\
		\hline
		W_{C_1} &   0 &   0   &  \Delta(2,3) & \Delta(2,3) \\
		W_{C_2} & \Delta(4,1) & \Delta(4,1) &     0 &  0  
	\end{array}
\end{equation*}
The set $\mathrm{Lip}_1^{\mathcal C}$ that in this special case is in bijection with $\mathrm{Lip}_1$ is
identified by the inequalities $-\Delta(2,3)\leq \lambda_1-\lambda_2\leq \Delta(4,1)$.
The FW-quasipotential is given by
\begin{equation*}
	\mathscr{FW}(1)=\mathscr{FW}(2)= \Delta(4,1)-\Delta(2,3);\qquad \mathscr{FW}(3)=\mathscr{FW}(4)=0.
\end{equation*}
By a straightforward calculation, we obtain that taking $\lambda_1=\Delta(4,1)-\Delta(2,3)$ and $\lambda_2=0$, we have $W_\lambda=\mathscr{FW}$. In Figure \ref{fig:cilindro} we represent the normalized FW-solution as a point in the set of normalized solutions $\mathcal{W}_0$ that is a subset of $\mathcal W$ that in this special case concides with $\mathrm{Lip}_1^{\mathcal C}$ and $\mathrm{Lip}_1$, since they are $|\mathcal C|=2$ dimensional polyhedrons.  

\begin{figure}
\begin{tikzpicture}\label{fig:cilindro}
	\draw[-stealth] (-5,0) -- (5,0) node[right]{$\lambda_1$};
	\draw[-stealth] (0,-3) -- (0,5) node[left]{$\lambda_2$};
	\draw[line width=0.7mm] (0,0) -- (0,2) node[left]{$\Delta(4,1)$};
	\draw[line width=0.7mm] (0,0) -- (0.2,0) node[above right]{$\mathscr{FW}$}  -- (1.5,0) node[below right]{$\Delta(2,3)$};
	\draw (-4.5,-1) -- (3,4);
	\draw (-3,-3) -- (4.5,2);
	\filldraw[black] (0.8,0) circle (2pt);
	\filldraw[black] (1.5,0) circle (2pt);
	\filldraw[black] (0,2) circle (2pt);
	\fill[gray!30,nearly transparent] (-4.5,-1) -- (3,4) -- (4.5,2) -- (-3,-3) -- cycle;

\end{tikzpicture}
\caption{The viscosity solutions of the HJ equation for the example in section \ref{exse}.}
\end{figure}

The grey area in figure \ref{fig:cilindro} is the set $\mathrm{Lip}_1^{\mathcal C}$ and $\mathcal{W}_0$ is indicated by the thiker segments on the axis.

\section{Conclusions and perspectives}
We end the paper with a short section illustrating possible developments and open problems. We discuss only general ideas and proposals.

\smallskip

We have focused on Markov chains satisfying Assumption \ref{assumption}, mainly for reasons of simplicity, since in this case the structure of the graph $(V,E_0)$ is simple and admits a canonical form. Many of our arguments also apply in the more general setting in which Assumption \ref{assumption} does not hold; however, in this case the structure of the graph $(V,E_0)$ may be arbitrary, representing the discrete counterpart of equilibrium points in the continuous setting. In particular, in the general case equation \eqref{1HJ} must be used, and it is not clear whether a simple and complete characterization of the solutions can always be obtained. 

\smallskip

Another interesting issue is the study of the asymptotic behavior of the invariant measure of a multiscale Markov chain using the Markov chain tree theorem. In this paper, we considered the case of exponentially decaying rates and focused only on the exponential asymptotic behavior. A possible extension is to analyze the full asymptotic behavior and to understand how HJ equations are related to it. This problem is closely connected to the renormalization group contruction studied in \cite{OS1, OS2, S1, S2, S3}.

\smallskip

The natural extension and generalization of our result is to consider metric graphs, with diffusions along the metric edges and possibly sticky behavior at the nodes, as in \cite{ACG}. The natural scaling to be considered is the small noise one, which may also involve, in  different ways, the behavior at the nodes. The resulting analytical problems fall within the general theory of HJ equations on metric graphs \cite{CM}, for which more delicate and involved definitions of solutions are necessary. Several scaling limits related to problems of this type appear to be of independent interest.

\smallskip

An open problem is to study the discrete HJ equation of the lifted Markov chain described in Appendix \ref{app:MC} and its relation to the discrete HJ equation on the original graph. It is reasonable to ask whether a simple continuous analogue of such a lifting exists.

\smallskip

Finally it is very interesting to investigate the application of the discrete HJ equations to problems involving metastability, like in \cite{OS1, OS2, S1, S2, S3}, and the relation with problems and algorithms for optimization of paths on graph \cite{Sch}.

\appendix

\section{Lifting Markov chains}\label{app:MC}
A simple and elegant way of proving the Markov chain tree Theorem is provided in \cite{AT} by constructing a Markov chain on an enlarged state space, such that its projection is the original Markov chain on $(V,E)$ with rates $r_N$. In particular, we consider the directed graph $\hat{\mathcal{G}}=(\mathcal T, \hat E)$. The set of vertices coincides with the set of arborescences and $(\tau,\tau')\in \hat E$ if and only if $\tau\in \mathcal T_x$, $\tau'\in \mathcal T_y$, $(x,y)\in E$ and the arborescence $\tau'$ is obtained by $\tau'=\left[\tau\cup (x,y)\right]\setminus (y,z)$ where $(y,z)$ is the unique edge exiting form $y$ in $\tau$. The rates $\hat r_N$ of $(\tau,\tau')\in \hat E$ with $\tau\in \mathcal T_x$ and $\tau'\in \mathcal T_y$ are given by $\hat r_N(\tau,\tau'):=r_N(x,y)$. Call $\hat X_N$ the Markov chain associated to such rates. There is a natural projection map $\Pi:(\mathcal T, \hat E) \to (V,E)$ such that, for any $\tau\in \mathcal T_x$, $\Pi(\tau)=x$ and, for any $(\tau,\tau')\in \hat E$ with $\tau\in \mathcal T_x$ and $\tau'\in \mathcal T_y$, $\Pi(\tau,\tau')=(x,y)$.
\begin{theorem}\cite{AT}
	The Markov chain $\hat X_N$ is irreducible, and has an unique invariant measure given by
	\begin{equation}\label{hat-inv}
		\hat \pi_N(\tau)=\frac{1}{Z_N}\mathcal P_{r_N}(\tau)\,,
	\end{equation}
	where $Z_N$ is a normalization factor and the symbol $\mathcal P$ has been defined in \eqref{wfw}.
	The process $X_N:=\Pi\left(\hat X_N\right)$, is a Markov chain on $(V,E)$ with transition rates $r_N$ and we deduce \eqref{eq:FWformula} as a consequence.
\end{theorem}
\begin{proof}
	We give a sketch of the proof.
	The irreducibility of the Markov chain $\hat X_N$ is discussed in \cite{AT}. The graph $(\mathcal T, \hat E)$ is Eulerian since, for each node $\tau$ there is a bijection among outgoing and ingoing edges. Indeed, for each $\tau\in \mathcal T_x$, the outgoing edges in $\hat{\mathcal G}$ are in correspondence with the outgoing edges from $x$ in $\mathcal G$. Let $\tau\in \mathcal T_x$ and $(\tau,\tau')\in \hat E$ with $\tau'\in \mathcal T_y$. We have that $\tau\cup (x,y)=\mathfrak u_x\in \mathcal U^{\mathrm{in}}_x$ so that there exists $\tau''\in \mathcal T_z$ such that $\tau''\cup (z,x)= \tau\cup (x,y)=\mathfrak u_x$. This correspondence $(\tau,\tau') \leftrightarrow (\tau'',\tau)$ gives the bijection among exiting and entering edges at $\tau$, i.e. a bijection between exiting and entering elements of $\hat E$ on each $\tau\in \mathcal T$ (see \cite{BC} section 2.6 for more details). Moreover by \eqref{hat-inv} we have, up to a common multiplicative constant, 
	\begin{equation}\label{staz-tree}
		\hat\pi_N(\tau)r_N(x,y)=\hat\pi_N(\tau'')r_N(z,x)= \mathcal P_{r_N}(\mathfrak u_x)\,.
	\end{equation}
	Using the fact that the graph $(\mathcal T, \hat E)$ is Eulerian, the stationarity of \eqref{hat-inv} follows. The last statements of the theorem are direct consequences of the definitions.
\end{proof}

{\bf Acknolegments:} 
M.A. acknowledges the Gruppo Nazionale per l’Analisi Matematica, la Probabilità e le loro Applicazioni (GNAMPA) of the Istituto Nazionale di Alta Matematica (INdAM). G.P. acknowledges the Gruppo Nazionale per la Fisica Matematica (GNFM) of the Istituto Nazionale di Alta Matematica (INdAM). D.G. acknowledges A. Faggionato, M. Mariani and Y. Shu for several useful discussions and the financial support from the Italian Research Funding Agency (MIUR) through PRIN project ``Emergence of condensation-like phenomena in interacting particle systems: kinetic and lattice models'', grant n. 202277WX43.



\begin{thebibliography}{99}

\bibitem{ACG} Aleandri M., Colangeli M., Gabrielli D.
\emph{A combinatorial representation for the invariant measure of diffusion processes on metric graphs}
ALEA Lat. Am. J. Probab. Math. Stat. \textbf{ 18}(2021), no.2, 1773–1799. \MR{4332219}

\bibitem{AT} Anantharam V.,Tsoucas P. \emph{A proof of the Markov chain tree theorem},
Statistics and Probability Letters \textbf{ 8} (1989), no. 2, 189–192. \MR{1017890}

\bibitem{JJ} Bang-Jensen J., Gutin G. \emph{Digraphs.} Theory, algorithms and applications. Second edition. Springer Monographs in Mathematics. Springer-Verlag London, Ltd., London, 2009. \MR{2472389}

\bibitem{BC} Biane P., Chapuy G.
\emph{Laplacian matrices and spanning trees of tree graphs.}
Ann. Fac. Sci. Toulouse Math. (6)   \textbf{ 26} (2017), no. 2, 235–261.  MR{3640890}

\bibitem{CM} Camilli F., Marchi C.
\emph{A comparison among various notions of viscosity solution for Hamilton-Jacobi equations on networks},
J. Math. Anal. Appl. \textbf{407} (2013), no. 1, 112–118. \MR{3063108}

\bibitem{dH} den Hollander F. \emph{Large deviations}
Fields Inst. Monogr., \textbf{14}
American Mathematical Society, Providence, RI, 2000. \MR{1739680}

\bibitem{E} Evans L.C. \emph{Weak KAM theory and partial differential equations, Calculus of variations and nonlinear partial differential equations} Lecture Notes in Math. \textbf{1927}, Springer, Berlin, (2008), pp. 123–154. \MR{2404075}

\bibitem{FG} Faggionato A., Gabrielli D. 
\emph{A representation formula for large deviations rate functionals of invariant measures on the one dimensional torus} Ann. Inst. Henri Poincaré Probab. Stat. \textbf{48}, (2012), no.1, 212–234. \MR{2919204}

\bibitem{F}  Fathi A. \emph{Weak kam theorem in lagrangian dynamics} to appear in Cambridge Studies in Advanced Mathematics. \href{http://assets.cambridge.org/97805218/22282/description/9780521822282_description.htm}{Link}

\bibitem{FW} Freidlin M.I., Wentzell A.D. \emph{Random perturbations of dynamical systems.} Third edition. Grundlehren der mathematischen Wissenschaften, \textbf{260.} Springer, Heidelberg, 2012. \MR{2953753}

\bibitem{GL} Gao Y., Liu J.G.
\emph{A selection principle for weak KAM solutions via Freidlin-Wentzell large deviation principle of invariant measures} SIAM J. Math. Anal. \textbf{55} (2023), no.6, 6457–6495. \MR{4662405}

\bibitem{GT} Graham R., Tel T. {\it Weak-noise limit of Fokker-Planck models and nondifferentiable potentials for dissipative dynamical systems} Phys. Rev. A \textbf{31}, 1109 (1985). \MR{0778512}

\bibitem{J} Jauslin H. R. {\it Nondifferentiable potentials for
	nonequilibrium steady states} Physica A \textbf{144}, 179 (1987). \MR{0899038}


\bibitem{No} Norris, J. R. \emph{Markov Chains} Cambrige University Press (1998). \MR{1600720}

\bibitem{OS1} Olivieri E., Scoppola E. \emph{Markov chains with exponentially small transition probabilities: first exit problem from a general domain. I. The reversible case.} J. Statist. Phys. \textbf{79} (1995), no. 3-4, 613–647. \MR{1327899}

\bibitem{OS2} Olivieri E., Scoppola E. \emph{Markov chains with exponentially small transition probabilities: first exit problem from a general domain. II. The general case.} J. Statist. Phys. \textbf{84} (1996), no. 5-6, 987–1041. \MR{1412076}

\bibitem{P} Pitman J. \emph{Combinatorial stochastic processes.} Lectures from the 32nd Summer School on Probability Theory held in Saint-Flour, July 7–24, 2002. With a foreword by Jean Picard. Lecture Notes in Mathematics, \textbf{1875} (2006). \MR{2245368}

\bibitem{PT} Pitman J., Tang, W. \emph{Tree formulas, mean first passage times and Kemeny’s constant of a Markov chain.} Bernoulli, \textbf{24} (3), 1942–1972 (2018). \MR{3757519}

\bibitem{Sch} Schrijver A. \emph{Combinatorial optimization. Polyhedra and efficiency. Vol. A. Paths, flows, matchings.} Chapters 1–38, Algorithms Combin., \textbf{24,A} Springer-Verlag, Berlin, (2003). \MR{1956924}  

\bibitem{S1}  Scoppola E. \emph{Renormalization group for Markov chains and application to metastability.} J. Statist. Phys. \textbf{73} (1993), no. 1-2, 83–121. \MR{1247859}

\bibitem{S2}  Scoppola E. \emph{Metastability for Markov chains: a general procedure based on renormalization group ideas. Probability and phase transition} (Cambridge, 1993), 303–322, NATO Adv. Sci. Inst. Ser. C: Math. Phys. Sci., 420, Kluwer Acad. Publ., Dordrecht, (1994). \MR{1283171}

\bibitem{S3} Scoppola E. \emph{Renormalization and graph methods for Markov chains. Advances in dynamical systems and quantum physics} (Capri, 1993), 260–281, World Sci. Publ., River Edge, NJ, (1995). \MR{1414703}


\bibitem{T} Tran H.V. \emph{Hamilton-Jacobi equations—theory and applications.}
Grad. Stud. Math., \textbf{213}, American Mathematical Society, (2021). \MR{4328923}
\end{thebibliography}




\end{document}